\documentclass[]{amsart}

\usepackage{graphicx,color,hyperref}
\numberwithin{equation}{section}

\usepackage[initials,nobysame]{amsrefs}

\newtheorem{theorem}{Theorem}[section]
\newtheorem{proposition}[theorem]{Proposition}
\newtheorem{corollary}[theorem]{Corollary}
\newtheorem{lemma}[theorem]{Lemma}
\theoremstyle{definition}
\newtheorem{definition}[theorem]{Definition}

\theoremstyle{remark}

\DeclareMathOperator{\sech}{sech}

\DeclareMathOperator{\am}{am}

\DeclareMathOperator{\Span}{span}

\usepackage{amssymb} 
\usepackage{color}
\usepackage{mathrsfs} 
\usepackage{url}

\newcommand{\R}{\mathbf{R}}

\newcommand{\N}{\mathbf{N}}

\newcommand{\B}{\mathcal{B}}
\newcommand{\K}{\mathrm{K}}

\newcommand{\Ap}{\mathcal{A}_\mathrm{pin}}

\newcommand{\amcn}{\am_{1,p}}

\begin{document}

\title{Stability of flat-core pinned $p$-elasticae}

\author{Tatsuya Miura}
\address[T.~Miura]{Department of Mathematics, Graduate School of Science, Kyoto University,
Kitashirakawa Oikawa-cho, Sakyo-ku, Kyoto 606-8502, Japan}
\email{tatsuya.miura@math.kyoto-u.ac.jp}

\author{Kensuke Yoshizawa}
\address[K.~Yoshizawa]{Faculty of Education, Nagasaki University, 1-14 Bunkyo-machi, Nagasaki, 852-8521, Japan}
\email{k-yoshizaw@nagasaki-u.ac.jp
}

\keywords{$p$-elastica, pinned boundary condition, stability, flat-core.}
\subjclass[2020]{49Q10 and 53A04.}

\date{\today}

\dedicatory{Dedicated to Professor Yoshikazu Giga,\\
with gratitude on the occasion of receiving the Kodaira Kunihiko Prize.}

\begin{abstract}
We classify the stability of flat-core $p$-elasticae in $\R^d$ subject to the pinned boundary condition. Together with previous work, this completes the classification of stable pinned $p$-elasticae in $\R^d$ for all $p\in(1,\infty)$ and $d\geq2$.
\end{abstract}

\maketitle

\section{Introduction}

Building upon the authors' previous work \cites{MY_Crelle, MY_JLMS}, in the present paper, we complete the classification of the stability of pinned $p$-elasticae in Euclidean space.

Let $p\in(1,\infty)$. 
For an immersed curve $\gamma \subset \mathbf{R}^d$, the \emph{$p$-bending energy} $\mathcal{B}_p$ is defined by 
\[
\mathcal{B}_p[\gamma]:=\int_\gamma |\kappa|^p \,ds,
\]
where $s$ denotes the arclength parameter and $\kappa:=\partial_s^2\gamma$ denotes the curvature vector. 
In general, a critical point of $\mathcal{B}_p$ under the fixed-length constraint is called a \emph{$p$-elastica}.
The classical case $p=2$ corresponds to Euler's elastica, and it is well known as a mathematical model of elastic rods.
Since Euler's work in the 18th century, the variational problem for $\mathcal{B}_2$ has been extensively studied.
More recently, there has also been a surge of interest in $\mathcal{B}_p$ for general exponents $p\in(1,\infty)$, and nowadays the study of the $p$-bending energy has branched into several directions
(see e.g.\ \cite{MM98, AM03, DFLM, nabe14, LP22, MY_AMPA, MY_IUMJ, MY_Crelle, MY_JLMS, AGM, BM04, BM07, MN13, AM17, FKN18,  NP20, OPW20, Poz20, SW20, BHV, BVH, Poz22, GPT23, OW23, DMOY} and references therein).

In this paper we focus on $p$-elasticae under the following pinned boundary condition.
Given $L>0$ and $P_0,P_1 \in \mathbf{R}^d$ with $|P_0-P_1|<L$, we consider $p$-elasticae in the admissible set 
\[
\mathcal{A}_{\rm pin}=\mathcal{A}_{\rm pin}(P_0,P_1,L):=
\{\gamma \in W^{2,p}_{\rm arc}(0,L;\mathbf{R}^d) \mid \gamma(0)=P_0, \ \gamma(L)=P_1 \}, 
\]
where $W^{2,p}_{\rm arc}(0,L;\mathbf{R}^d)$ denotes the set of arclength parametrized $W^{2,p}$-curves of length $L$, i.e., 
\[
W^{2,p}_{\rm arc}(0,L;\mathbf{R}^d) := \{\gamma \in W^{2,p}(0,L;\mathbf{R}^d) \mid \,\! |\gamma'|\equiv1 \}. 
\]
A critical point of $\B_p$ in $\mathcal{A}_{\rm pin}$ is called a \emph{pinned $p$-elastica} (see Section~\ref{subsect:p-elastica} and also \cite{MY_IUMJ} for details).
For a general $p\in(1,\infty)$, the complete classification of pinned $p$-elasticae has been already obtained; see \cite[Theorem 1.1]{MY_IUMJ} in the planar case and \cite[Theorem 1.6]{GMarXiv2501} in general dimensions.
We say that a pinned $p$-elastica in $\mathcal{A}_{\rm pin}$ is \emph{stable} if it is a local minimizer of $\B_p$ in $\mathcal{A}_{\rm pin}$.

The aim of this manuscript is to fully characterize the stability of pinned $p$-elasticae in Euclidean space.
To the best of the authors' knowledge, the only known results on the stability of $p$-elasticae are due to Gruber--P\'ampano--Toda in the $2$-sphere \cite{GPT23} and to the authors in the plane \cites{MY_Crelle, MY_JLMS}.
While the difference in the ambient space is apparent, these studies also differ in other essential aspects.
More precisely, in \cite{GPT23}, Gruber--P\'ampano--Toda showed the instability of closed smooth free $p$-elasticae in the $2$-sphere for $p\in(0,1)$, where ``free'' means the absence of the length constraint.
Their proof is based on finding a suitable variation in the negative direction of the second variation of $\mathcal{B}_p$.
The validity of this method partly relies on their smooth regularity assumption, ruling out possible less regular $p$-elasticae (cf.\ \cite{SW20}).
On the other hand, in \cites{MY_Crelle, MY_JLMS}, the authors treated closed and pinned planar $p$-elasticae in the $W^{2,p}$-Sobolev class for $p\in(1,\infty)$.
In particular, the fixed-length constraint makes it more difficult to find energy-decreasing variations. 
Furthermore, the regularity of pinned $p$-elasticae is generically lost \cite{MY_IUMJ}, while the classical second variation of $\mathcal{B}_p$ may not exist on non-smooth $p$-elasticae (cf.\ \cite[Appendix A]{MY_Crelle}).
For these reasons, in \cite{MY_Crelle}, the authors recently developed a new approach which we call the ``cut-and-paste'' trick, not relying on the second variation but the geometric invariance of the functional.
This paper further develops the cut-and-paste approach.

The stability of many pinned $p$-elasticae has already been clarified by previous works.
Roughly speaking, the only remaining case is the class of \emph{flat-core $p$-elasticae}, first introduced by Watanabe \cite{nabe14} in the planar case.
Recall that a flat-core $p$-elastica is defined as a concatenation of certain loops and segments (see Figure~\ref{fig:quasi_all} and \cite{nabe14,MY_AMPA}, as well as \cite{GMarXiv2501} in general dimensions): such critical points appear if and only if $p>2$ and, under the pinned boundary condition, additionally satisfy $|P_0-P_1|\geq\frac{1}{p-1}L$ \cite[Theorem 1.1]{MY_IUMJ}.
More precisely, in the planar case, the authors showed that a pinned $p$-elastica is unstable if it is neither a (unique) global minimizer nor a flat-core $p$-elastica \cite[Corollary 2.14]{MY_Crelle}.
In addition, Gruen and the first author showed that any non-planar pinned $p$-elastica is of flat-core type and has larger energy than the planar global minimizer \cite[Theorem 1.6, Corollary 1.7]{GMarXiv2501}.
However, the stability of flat-core pinned $p$-elasticae still leaves several cases open, and resolving these is the main result of this paper.

Now we formulate our main results.
We first focus on the planar case.
The stability of planar flat-core $p$-elasticae is particularly delicate since it strongly depends on \emph{how the loops and segments are concatenated}.
Indeed, according to \cite[Theorem 2.13]{MY_Crelle}, a flat-core pinned $p$-elastica is unstable in $\R^2$ if it is not \emph{quasi-alternating}, that is, if either there is a loop touching an endpoint as in Figure~\ref{fig:quasi_all} (i), or there are adjacent loops in opposite directions as in Figure~\ref{fig:quasi_all} (ii) (see also Definition~\ref{def:quasi-alternating}).
On the other hand, it is shown in \cite{MY_JLMS} that a flat-core $p$-elastica is stable in $\R^2$ if it is \emph{alternating}, i.e., any segment does not touch an endpoint nor lie between any two loops as in Figure~\ref{fig:quasi_all} (iii) (see also Proposition~\ref{prop:alternating}). 
The alternating class can naturally be regarded as a generic subclass of the quasi-alternating class.
However, the remaining case of quasi-alternating but not alternating flat-core $p$-elasticae as in Figure~\ref{fig:quasi_all} (iv) was an open problem (see \cite{MY_JLMS}*{Problem 1.4}).

\begin{figure}[htbp]
\centering
\includegraphics[width=100mm]{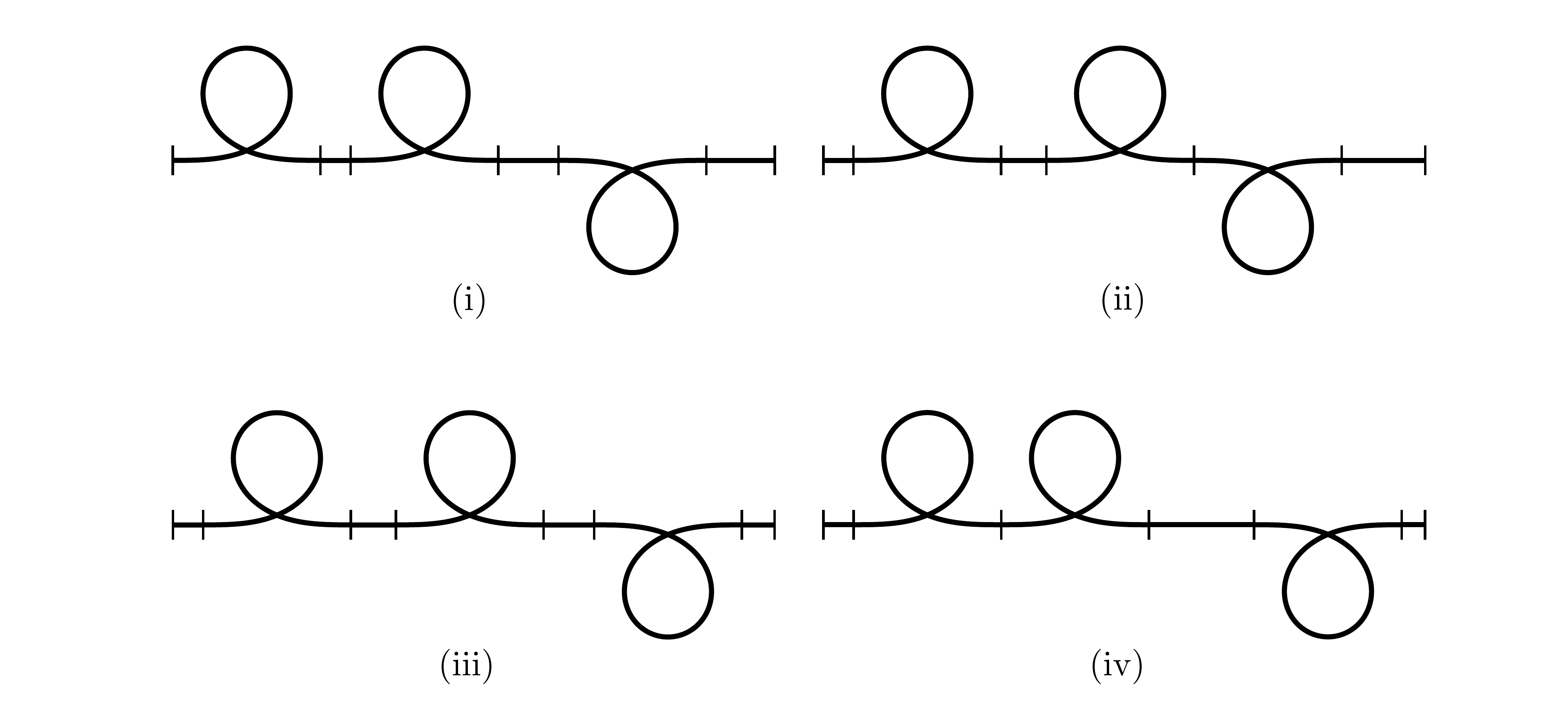} 
\caption{Examples of flat-core pinned $p$-elasticae.
(i), (ii) Not quasi-alternating (unstable \cite{MY_Crelle}).
(iii) Alternating (stable \cite{MY_JLMS}).
(iv) Quasi-alternating (stable, Theorem~\ref{thm:2D-quasi}).
[The figure is reproduced from \cite[Figure 1]{MY_JLMS}, with permission from Wiley.]
}
\label{fig:quasi_all}
\end{figure}

Our first main result resolves this last-standing open problem.

\begin{theorem}\label{thm:2D-quasi}
Any quasi-alternating flat-core pinned $p$-elastica is stable in $\R^2$.
\end{theorem}

The proof is based on a simple geometric maneuver.
If there are two adjacent loops in the same direction, then we can always insert a horizontal segment into the perturbation (see Figure~\ref{fig:2D-quasi}).
This procedure essentially reduces the problem to the alternating case, where the stability is already known.

As discussed above, Theorem \ref{thm:2D-quasi} together with previous work provides a complete characterization of the stability of stable pinned $p$-elasticae in $\R^2$.

\begin{corollary}\label{cor:2D-stability}
    Let $d=2$ and $\gamma\in\mathcal{A}_{\rm pin}$. Then $\gamma$ is a local minimizer of $\B_p$ in $\mathcal{A}_{\rm pin}$ if and only if $\gamma$ is either a global minimizer or a quasi-alternating flat-core pinned $p$-elastica.
\end{corollary}

Recall that the latter case can occur if and only if $p>2$ and $|P_0-P_1|>\frac{1}{p-1}L$, and 
otherwise, the only local minimizers are global minimizers (unique up to reflection if $|P_0-P_1|>0$, while up to reflection and rotation if $|P_0-P_1|=0$).
See also \cite[Remark 2.10]{MY_JLMS} for the slightly delicate critical case $|P_0-P_1|=\frac{1}{p-1}L$. 

Now we turn to the stability of non-planar pinned $p$-elasticae.
Our second main result asserts that, in stark contrast to the planar case, all flat-core pinned $p$-elasticae are unstable in higher codimensions. 

\begin{theorem}\label{thm:nD-stability}
Any flat-core pinned $p$-elastica is unstable in $\R^d$ for $d\geq3$.
\end{theorem}

The proof is divided into two cases: if a segment touches an endpoint, then we construct a new, non-planar perturbation (see Figure \ref{fig:3D-quasi}) to bring the problem into the framework of the cut-and-paste trick \cite{MY_Crelle}; if a loop touches an endpoint, then our previous argument in the planar case \cite[Proposition 6.4]{MY_Crelle} directly extends to the non-planar problem.

As a consequence of Theorem \ref{thm:nD-stability} and previous work, in higher codimensions, any local minimizer must be a global minimizer, which is thus planar and unique up to rotation.
\begin{corollary}\label{cor:nD-stable-unique}
    Let $d\geq3$ and $\gamma\in\mathcal{A}_{\rm pin}$.
    Then $\gamma$ is a local minimizer of $\B_p$ in $\mathcal{A}_{\rm pin}$ if and only if $\gamma$ is a global minimizer.
\end{corollary}
Theorems~\ref{thm:2D-quasi} and \ref{thm:nD-stability} reveal a peculiar stability-transition phenomenon between 2D and 3D: planar, pinned, quasi-alternating flat-core $p$-elasticae are stable in $\mathbf{R}^2$ but unstable in $\mathbf{R}^3$. 
Such 2D/3D transition phenomena frequently occur in elastica problems.
Langer--Singer \cite{LS_85} showed that spatially stable closed elasticae are uniquely given by singly-covered circles, which in particular implies that multiply-covered circles are unstable in $\R^3$, while they are stable in $\mathbf{R}^2$ (see also \cite{MY_Crelle}*{Corollary 7.3} for the $p$-counterpart).
In addition, as shown in \cite{Miura_LiYau}, there may not exist a closed curve in $\R^2$ which attains equality in a Li--Yau type inequality (providing a lower bound for the normalized bending energy in terms of multiplicity), while in $\R^3$ for any multiplicity there is a closed curve attaining equality (see also \cite{GMarXiv2501} for the $p$-counterpart).

This paper is organized as follows. 
In Section~\ref{sect:preliminary} we prepare notation and recall the definition and some properties of flat-core pinned $p$-elasticae. 
We then prove Theorem~\ref{thm:2D-quasi} in Section~\ref{sect:2D}, and Theorem~\ref{thm:nD-stability} in Section~\ref{sect:3D}.

\subsection*{Acknowledgments}
The first author is supported by JSPS KAKENHI Grant Numbers JP23H00085, JP23K20802, and JP24K00532.
The second author is supported by JSPS KAKENHI Grant Number JP24K16951.

\section{Preliminary}\label{sect:preliminary}

We introduce the concatenation of curves $\gamma_1:[a_1,b_1] \to \R^d$ with $L_1=b_1-a_1$ and $\gamma_2:[a_2,b_2] \to \R^d$ with $L_2=b_2-a_2$ by
\begin{equation*}
    (\gamma_1 \oplus \gamma_2)(s) := \begin{cases}
        \gamma_1(s+a_1), & s \in [0,L_1],\\
        \gamma_2(s+a_2-L_1)+ \gamma_1(b_1)-\gamma_2(a_2), & s \in [L_1,L_1+L_2].
    \end{cases}
\end{equation*}
We also define inductively $\bigoplus_{j=1}^N \gamma_j= \gamma_1 \oplus \dots \oplus \gamma_N := (\gamma_1 \oplus \dots \oplus \gamma_{N-1})\oplus \gamma_N$.
In addition, for $P\in \R^d$ and $\gamma:[a,b] \to \R^d$ we define $(P\oplus\gamma)(s):=P+\gamma(s)-\gamma(a)$ for $s\in[a,b]$.

\subsection{$p$-Hyperbolic functions}

To begin with, we recall $p$-hyperbolic functions (only for $p>2$) introduced in \cite{MY_AMPA}.
For $p>2$, we define
\[
\mathrm{F}_{1,p}(x,1):=\displaystyle \int_0^{x} \frac{d\phi}{|\cos \phi|^{\frac{2}{p} } }, \quad x\in\R,
\]
and
\begin{align*} 
\K_{p} (1):= \mathrm{F}_{1,p}(\tfrac{\pi}{2},1)=
\displaystyle \int_0^{\frac{\pi}{2}} \frac{d\phi}{(\cos \phi)^{\frac{2}{p} } } < \infty.  
\end{align*}
Here, $\mathrm{F}_{1,p}(x,1)$ (resp.\ $\K_{p}(1)$) corresponds to the incomplete (resp.\ complete) $p$-elliptic integral of the first kind of modulus $q=1$ in \cite{MY_AMPA}*{Definition 3.1}.

\begin{definition}[$p$-Hyperbolic functions] \label{def:p-sech}
Let $p>2$.
We define $\amcn(x,1)$ by the inverse functions of $\mathrm{F}_{1,p}(x,1)$, i.e., for $x\in\R$, 
\begin{align} \nonumber
x= \int_0^{\amcn(x,1)} \frac{1}{ (\cos\phi)^{\frac{2}{p}} }\,d\phi.
\end{align}
The \emph{$p$-hyperbolic secant} $\sech_p x$ is defined by
\begin{align}\label{eq:sech_p} 
 \sech_p x:= 
\begin{cases}
 (\cos \amcn(x,1))^{\frac{2}{p}}, \quad & x\in (-\K_p(1), \K_p(1)), \\
 0, &x\in \R \setminus (-\K_p(1), \K_p(1)).
\end{cases}
\end{align}
The \emph{$p$-hyperbolic tangent} $\tanh_p x$ is defined by
\[\tanh_p x:=\int_0^x (\sech_p t)^p dt, \quad x\in \R. \]
\end{definition}

It is already known in \cite{MY_AMPA}*{Proposition 3.10} that $\sech_p $ is an even nonnegative continuous function on $\R$ and,  in $[0, \K_{p}(1))$, strictly decreasing from $1$ to $0$. 
Moreover, a straightforward calculation yields
\begin{align}\label{eq:diff^2-sech_p}
\frac{d^2}{dx^2}(\sech_p{x})^{p-1} = -2 \frac{p-1}{p} \sech_p{x} \big( 2 (\sech_p{x})^p -1 \big).
\end{align}
We also have
\begin{align}\label{eq:diff^2-tanh_p}
\frac{d^2}{dx^2}(\tanh_p{x})=-2\sin\amcn(x,1)\cos^{1+\frac{2}{p}}\amcn(x,1).   
\end{align}

\subsection{$p$-Elastica}\label{subsect:p-elastica}
In this paper, following \cite{GMarXiv2501}*{Definition 1.5}, we call $\gamma\in \mathcal{A}_{\rm pin}$ a \emph{pinned $p$-elastica} if there exists $\lambda\in\R$ such that
\begin{align}\label{eq:def-pinned_p-elastica}
\frac{d}{d\varepsilon} \Big(\B_p[\gamma+\varepsilon\eta]+ \lambda \mathcal{L}[\gamma+\varepsilon\eta]\Big)\Big|_{\varepsilon=0}=0
\end{align}
for all $\eta\in C^\infty_0(0,L;\mathbf{R}^d)$.
Here $\mathcal{L}$ denotes the length functional $\mathcal{L}[\gamma]:=\int_\gamma ds$, and the class $C^\infty_0$ means $C^\infty$ functions vanishing at the endpoints.
Thanks to the multiplier method, this is the standard first-order necessary condition for $\gamma\in\Ap$ to be a local minimizer of $\B_p$ in $\Ap$ (see \cite{MY_IUMJ} for details).
More generally, $\gamma\in W^{2,p}_{\rm arc}(0,L;\R^d)$ is called a \emph{$p$-elastica} if \eqref{eq:def-pinned_p-elastica} holds for all $\eta\in C^\infty_{\rm c}(0,L;\mathbf{R}^d)$ (with compact support).

Now we recall the definition of flat-core $p$-elasticae more precisely.
To this end we first define auxiliary curves: 
Let $\{e_1,\dots,e_d\}$ be the canonical basis of $\R^d$.
For $L \geq0$ we define $\gamma_\ell^{L}:[0,L]\to\R^d$ by
\[
\gamma_\ell^{L}(s) := -se_1,
\]
a line segment of length $L$.
For $p>2$ and $\sigma \in \mathbf{S}^{d-2} \subset \Span\{e_2,\dots,e_d\} \subset \R^d$, we define $\gamma_b^{\sigma}:[-\K_p(1),\K_p(1)] \to \R^d$ by 
\begin{align}    
   \gamma_b^{\sigma}(s) &:= (2 \tanh_p s -s)e_1+ \tfrac{p}{p-1} (\sech_p s)^{p-1}\sigma, \label{def:gamma_b}
\end{align}
a ``borderline-type'' finite-length loop contained in the plane spanned by $e_1$ and $\sigma$.
Then, an alternative concatenation of the above segments $\gamma_\ell^{L}$ (allowing $L=0$) and loops $\gamma_b^{\sigma}$ is called a flat-core $p$-elastica (see also \cite[Definition 3.9]{GMarXiv2501}).

In this paper we are interested in flat-core $p$-elasticae arising as pinned $p$-elasticae.
By \cite{GMarXiv2501}*{Theorem 1.6} and \cite{MY_IUMJ}*{Theorem 1.1} we know that there do exist pinned $p$-elasticae of flat-core type. 
To be more precise, $\gamma \in \Ap$ is a \emph{flat-core pinned $p$-elastica} if $p>2$ and $|P_1-P_0|\in[\frac{1}{p-1}L,L)$, and
if there are $N\in \N$, $r\in[\frac{1}{p-1},1)$, $\sigma_1, \dots, \sigma_N \in \mathbf{S}^{d-2}\subset \Span\{e_2,\dots,e_d\}$, and $L_1, \ldots, L_{N+1}\geq0$ such that a similar transformation and reparametrization of $\gamma$ is given by
    \begin{align} \label{eq:N-loop-flat-core}
        \gamma_{\mathrm{flat}}:=\bigg( \bigoplus_{j=1}^N \big( \gamma_{\ell}^{L_j} \oplus \gamma_{b}^{\sigma_j} \big) \bigg) \oplus \gamma_{\ell}^{L_{N+1}}
    \end{align}
and in addition, the numbers $p, r, N$ and $(L_1, \ldots, L_{N+1})$ satisfy  
    \begin{align} \label{eq:sum-flatparts}
    \sum_{j=1}^{N+1}L_j = 2N\frac{r-\frac{1}{p-1}}{1-r} \K_{p}(1).
    \end{align}
Note that if $\gamma\in\Ap$ is a flat-core pinned $p$-elastica as above, then we need to have $r=\frac{|P_0-P_1|}{L}$.
On the other hand, $(\sigma_1,\ldots,\sigma_N)$ is arbitrary, and also $N$ and $(L_1, \ldots, L_{N+1})$ are arbitrary whenever \eqref{eq:sum-flatparts} holds.
In addition, any curve $\gamma_{\mathrm{flat}}$ of the form \eqref{eq:N-loop-flat-core} satisfies
\begin{align}
    & \mathcal{L}[\gamma_{\mathrm{flat}}]=\bar{L}:=2N\K_p(1)+\sum_{j=1}^{N+1} L_j, \label{eq:model_length}\\
    & \gamma_{\mathrm{flat}}(\bar{L}) - \gamma_{\mathrm{flat}}(0) = -\bigg(\frac{2N}{p-1}\K_{p}(1) + \sum_{j=1}^{N+1} L_j \bigg)e_1. \label{eq:model_distance}
\end{align}

Finally, we discuss geometric properties of each loop which we will use later.
The tangent direction of each loop $\gamma_b^{\sigma}$ at the vertex is given by 
\begin{align}\label{eq:loop_tangent_top}
    \big(\gamma_b^{\sigma}\big)'(0) = e_1,
\end{align}
which follows from \eqref{def:gamma_b} combined with the fact that $(\tanh_p)'(0)=\sech_p 0=1$ and that $(\sech_p)'(0)=0$ as $\sech_p$ is even. 
On the other hand, the tangent directions at the endpoints are
\begin{align}\label{eq:loop_tangent_ends}
    \big(\gamma_b^{\sigma}\big)'(\pm\K_p(1)) = -e_1,
\end{align}
which follows from \eqref{def:gamma_b} combined with the direct computations
\begin{equation*}
    (\sech_p^{p-1})'(\pm\K_p(1))=0, \quad (\tanh_p)'(\pm\K_p(1))=\sech_p^p(\pm\K_p(1))=0.
\end{equation*}
Furthermore, using \eqref{eq:diff^2-sech_p} and \eqref{eq:diff^2-tanh_p}, we deduce from \eqref{def:gamma_b} that the curvature vector $(\gamma_b^\sigma)''$ satisfies 
\begin{align}\label{eq:loop_curvature_vector}
    (\gamma_b^\sigma)''(\pm\K_p(1))=0 \quad \text{and} \quad (\gamma_b^\sigma)''(0)=-2\sigma.
\end{align}

\section{The stability of planar quasi-alternating flat-core}\label{sect:2D}

In this section we focus on the planar case. 
The stability of planar $p$-elasticae is better understood than the non-planar case, and indeed it is already shown in \cite{MY_Crelle}*{Corollary 2.14} that every stable pinned $p$-elastica must be either a global minimizer or a quasi-alternating flat-core $p$-elastica.
Here let us precisely recall the definition of the quasi-alternating class (see Figure~\ref{fig:quasi_all} (iv)).

\begin{definition}[Quasi-alternating]\label{def:quasi-alternating}
Let $d=2$.
Let $\gamma \in \Ap$ be a (planar) flat-core pinned $p$-elastica. 
Let $N\in \mathbf{N}$, $\{\sigma_j\}_{j=1}^N \subset\mathbf{S}^0=\{+e_2,-e_2\}$, and $L_1,\ldots, L_{N+1}\geq0$ be such that $\gamma$ is of the form \eqref{eq:N-loop-flat-core} up to similarity and reparametrization. 
We say that $\gamma$ is \emph{quasi-alternating} if the following two conditions hold: 
\begin{itemize}
\item[(i)] $L_1>0$ and $L_{N+1} >0$.
\item[(ii)] For $j\in\{2,\ldots, N\}$, if $L_j=0$, then $\sigma_j = \sigma_{j-1}$. 
\end{itemize}
\end{definition}

The aim of this section is to prove that any quasi-alternating flat-core planar $p$-elastica is a local minimizer, which together with \cite{MY_Crelle}*{Corollary 2.14} ensures Theorem~\ref{thm:2D-quasi}.
As already mentioned, a quasi-alternating flat-core pinned planar $p$-elastica is known to be stable if it is alternating \cite{MY_JLMS}*{Theorem 1.3}. 
Here, we say that a flat-core $p$-elastica $\gamma \in \Ap$ is \emph{alternating} if, up to similarity and reparametrization, $\gamma$ is of the form \eqref{eq:N-loop-flat-core} with $L_1, \ldots, L_{N+1}>0$. 
More precisely, we already know the following

\begin{proposition}
\label{prop:alternating}
Let $d=2$.
If $\gamma\in \Ap$ is an alternating flat-core pinned $p$-elastica, then $\gamma$ is stable in $\Ap$. 
\end{proposition}

Now we prove Theorem~\ref{thm:2D-quasi}.

\begin{proof}[Proof of Theorem~\ref{thm:2D-quasi}]
Let $\gamma\in \Ap$ be a quasi-alternating flat-core pinned $p$-elastica. 
Up to similarity and reparametrization, we may assume that $\gamma$ is of the form \eqref{eq:N-loop-flat-core}: just for geometrical simplicity, here we consider its reflection
\begin{align}\label{eq:alternating-model}
        \gamma = R_\pi \bigg( \Big( \bigoplus_{j=1}^N \big(\gamma_\ell^{L_j} \oplus \gamma_b^{\sigma_j}\big) \Big) \oplus \gamma_\ell^{L_{N+1}} \bigg),
    \end{align}
with some $N\in \N$, $\{\sigma_j\}_{j=1}^N\subset\{e_2,-e_2\}$, and $(L_1, \ldots, L_{N+1})\in[0,\infty)^{N+1}$, where $R_\pi$ denotes the rotation matrix through angle $\pi$.
In this case, we have $\gamma\in \Ap(P_0,P_1, L)$ with $P_0=0$, $P_1=\ell e_1$, $\ell:= \frac{2N}{p-1}\K_p(1)+\sum_{j=1}^{N+1} L_j$, and $L:=2N\K_p(1) + \sum_{j=1}^{N+1} L_j$, cf.\ \eqref{eq:model_length} and \eqref{eq:model_distance}.

We now define the set of indices satisfying Definition~\ref{def:quasi-alternating} (ii), that is,
\[
J:=\{ j\in\{2, \ldots, N\} \mid L_j=0, \ \sigma_{j-1}=\sigma_j \} 
\]
and write $J=\{j_1, \ldots, j_{M}\}$ with $M=\#J$.
We may assume that $J\neq \emptyset$, i.e., $M\geq1$, since if otherwise then $\gamma$ is stable by Proposition~\ref{prop:alternating}.
Choose a partition $\{s_i\}_{i=0}^{M+1}$ of $[0,L]$ in such a way that  
\begin{equation*}
    \text{$s_0=0$, $s_{M+1}=L$, and $s_{i}=2(j_i-1)\K_p(1)+\sum_{\nu=1}^{j_i} L_\nu$ for $i\in\{1,\ldots,M\}$,}
\end{equation*}
which corresponds to the joint of two adjacent loops (see Figure~\ref{fig:2D-quasi}).
Then, in view of \eqref{eq:alternating-model} and \eqref{eq:loop_tangent_ends}, we see that for all $i\in \{1,\ldots, M\}$
\begin{align}\label{eq:direction_is_e_1}
    \gamma'(s_i) = R_\pi \big(\gamma_b^{\sigma_{j_i}}\big)'(\K_p(1)) = e_1.
\end{align}
Recall from \cite{MY_AMPA}*{Theorem 1.3} that the signed curvature $k_b^{\pm}$ of $\gamma_b^{\pm e_2}$ is given by $k_b^{\pm}(s)=\pm2\sech_p{s}$.
Combining this with $\sigma_{j_i}=\sigma_{j_i-1}$, we find that the signed curvature $k$ of $\gamma$ satisfies $k(s_i)=0$ and, if $\sigma_{j_i}=e_2$ (resp.\ $\sigma_{j_i}=-e_2$), then $k(s)>0$ (resp.\ $k(s)<0$) for all $0<|s-s_i|<\K_p(1)$. 
Moreover, this implies that 
\begin{align}\label{eq:partition}
\text{the tangential angle $\theta$ of $\gamma$ is strictly monotone in a neighborhood of } s_{i}
\end{align}
for all $i\in \{1,\ldots, M\}$. 

Now we prove the stability of $\gamma$ by contradiction.
Suppose that $\gamma$ is not a local minimizer of $\B_p$ in $\Ap(P_0,P_1,L)$.
Then there exists $\{\gamma_n\}_{ n\in\mathbf{N}} \subset \Ap(P_0,P_1,L)$ such that $\gamma_n\to\gamma$ in $W^{2,p}(0,L;\R^2)$, and hence also in $C^1([0,L];\R^2)$, and 
\begin{align}\label{eq:energy_gamma_n}
\B_p[\gamma_n] < \B_p[\gamma] \quad \text{for all }\  n\in\mathbf{N}.
\end{align} 
Combining the $C^1$-convergence with \eqref{eq:direction_is_e_1} and \eqref{eq:partition}, for each $i\in \{1,\ldots,M\}$ we can pick a sequence $\{s_{i,n}\}_{n\in\mathbf{N}}\subset[0,L]$ such that (as in Figure~\ref{fig:2D-quasi})
\begin{align}\label{eq:partition_n}
\gamma_n'(s_{i,n})=e_1, \quad \lim_{n\to\infty} s_{i,n}=s_{i}.
\end{align}
Taking $s_{0,n}:=0$ and $s_{M+1,n}:=L$, for all large $n$ we can define a new partition $\{s_{i,n}\}_{i=0}^{M+1}$ of $[0,L]$.
Define $\Gamma_n:[0,L+M]\to\R^2$ by
\[
\Gamma_n:= \bigg(\bigoplus_{i=1}^M\big( \gamma_n|_{[s_{i-1,n},s_{i,n}]} \oplus \bar{\gamma} \big)\bigg)\oplus \gamma_n|_{[s_{M,n},s_{M+1}]}, 
\]
where $\bar{\gamma}:[0,1]\to\R^2$ is the segment $\bar{\gamma}(s)=s e_1$ (see Figure~\ref{fig:2D-quasi}).
From \eqref{eq:partition_n} we see that $\Gamma_n$ is of class $C^1$.
Combining this with the fact that each $\gamma_n|_{[s_{i-1,n}, s_{i,n}]}$ and $\bar{\gamma}$ are of class $W^{2,p}$, 
we see that $\Gamma_n$ is of class $W^{2,p}$. 
Therefore, we have $\Gamma_n \in \Ap(P_0, P_1', L+M)$ with $P_1':=P_1+Me_1$. 
In addition, since $\bar{\gamma}''\equiv0$,
\begin{align}\notag 
    \B_p[\Gamma_n] = \|\Gamma_n''\|_{L^p(0,L+M)} = \|\gamma_n''\|_{L^p(0,L)}= \B_p[\gamma_n].
\end{align}
This together with \eqref{eq:energy_gamma_n} yields 
\begin{align}\label{eq:energy_Gamma_n}
   \B_p[\Gamma_n] <\B_p[\gamma].
\end{align}
Furthermore, by construction it is easy to check that 
\begin{align}\label{eq:Gamma_n_to_Gamma}
   \Gamma_n \to \Gamma \quad \text{in}\ \  W^{2,p}(0,L+M;\mathbf{R}^2),
\end{align}
where
\[
\Gamma:= \bigg(\bigoplus_{i=1}^M\big( \gamma|_{[s_{i-1},s_{i}]} \oplus \bar{\gamma} \big)\bigg)\oplus \gamma|_{[s_{M},s_{M+1}]}.
\]
By construction, $\Gamma$ is an alternating flat-core $p$-elastica in $\Ap(P_0, P_1', L+M)$, and satisfies $\B_p[\Gamma]=\B_p[\gamma]$.  
Thus by Proposition~\ref{prop:alternating} it follows that $\Gamma$ is a local minimizer of $\B_p$ in $\Ap(P_0,P_1',L+M)$.
This together with \eqref{eq:Gamma_n_to_Gamma} yields $\B_p[\Gamma_n]\geq \B_p[\Gamma]=\B_p[\gamma]$ for all (large) $n\in\mathbf{N}$, which contradicts \eqref{eq:energy_Gamma_n}.
\end{proof}

\begin{figure}[htbp]
\centering
\includegraphics[width=100mm]{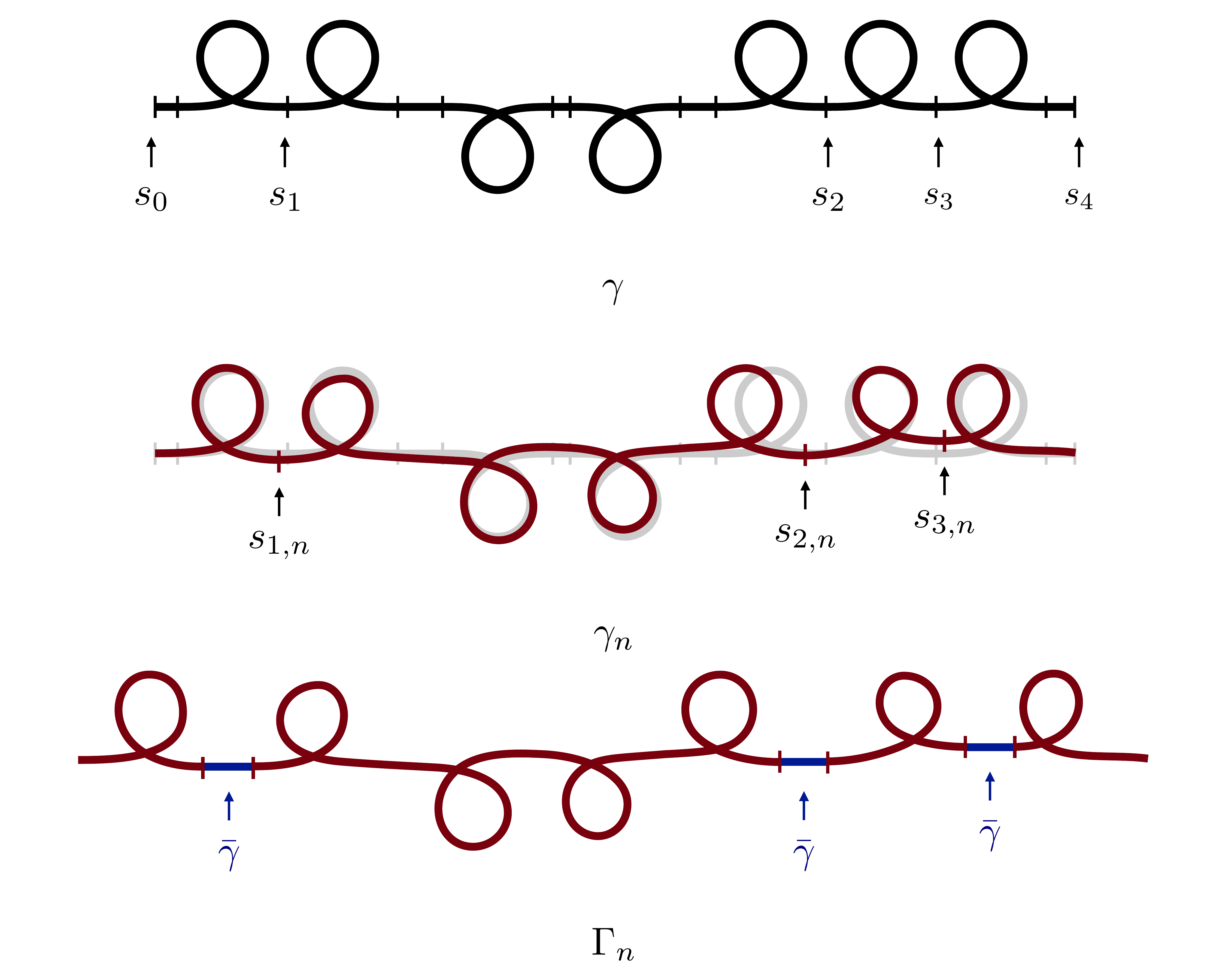} 
\caption{A quasi-alternating flat-core pinned $p$-elastica and its partition $\{s_i\}_{i=0}^4$ (top).
A perturbation $\gamma_n$ of $\gamma$ and its partition $\{s_{i,n}\}_{i=1}^3$ (middle).
A construction of $\Gamma_n$ by inserting line segments $\bar{\gamma}$ at the partition points (bottom). 
}
\label{fig:2D-quasi}
\end{figure}

\begin{proof}[Proof of Corollary~\ref{cor:2D-stability}]
It follows from Theorem~\ref{thm:2D-quasi} and \cite{MY_Crelle}*{Corollary 2.14}. 
\end{proof}

\section{The instability of quasi-alternating flat-core in higher dimensions}\label{sect:3D}

In this section we discuss the non-planar case, so assume $d\geq3$ throughout. 

To begin with, 
recall from \cite{GMarXiv2501}*{Theorem 1.3} that the any pinned $p$-elastica is of class $C^2$. 
Therefore, we can directly extend (the contrapositive of) a necessary condition for the local minimality of pinned planar $p$-elasticae \cite[Lemma 4.2]{MY_Crelle} from the planar to non-planar case as follows. 

\begin{lemma}\label{lem:Crelle-3D}
Let $\gamma \in\Ap$ and suppose that there exists a sequence $\{\gamma_n\}_{n\in\N} \subset \Ap$ satisfying the following three conditions: 
\begin{align}\tag{C}\label{eq:C}
\begin{cases}
\mathrm{(i)} \ \ \ \ \gamma_n \to \gamma  & \text{in} \ \ W^{2,p}(0,L;\R^d) \ \  \text{as} \ \  n\to \infty, \\
\mathrm{(ii)} \ \ \ \, \mathcal{B}_p[\gamma_n] \leq \mathcal{B}_p[\gamma] &  \text{for all (large)}\quad n\in \N, \\
\mathrm{(iii)} \ \ \, \gamma_n \notin C^2(0,L;\R^d) &\text{for all (large)} \quad n\in \N.
\end{cases}
\end{align}
Then, $\gamma$ is unstable in $\Ap$.
\end{lemma}

We safely omit the proof of Lemma~\ref{lem:Crelle-3D} as it is completely parallel to the argument for \cite[Lemma 4.2]{MY_Crelle}.

This kind of criterion is very useful in terms of the stability of non-planar flat-core pinned $p$-elasticae. Indeed, here we construct a new (non-planar) perturbation to obtain the following

\begin{proposition}\label{prop:instability-quasi-3D}
Let $d\geq3$ and $\gamma\in \Ap$ be a flat-core pinned $p$-elastica. 
Suppose that, up to similarity and reparametrization, $\gamma$ is of the form \eqref{eq:N-loop-flat-core} with $L_1>0$.
Then $\gamma$ is unstable in $\Ap$.
\end{proposition}
\begin{proof}
Without loss of generality we may assume that $\gamma$ is exactly given by \eqref{eq:N-loop-flat-core} with $L_1>0$.
Up to rotation we may assume that $\sigma_1=e_2$. 

Let $s_0:=L_1 + \K_p(1)$, corresponding to the vertex of the first loop, and consider the decomposition $\gamma=\gamma|_{[0,s_0]} \oplus \gamma|_{[s_0,L]}$.
Then by \eqref{eq:loop_tangent_top} we have $\gamma'(s_0)=(\gamma_b^{\sigma_1})'(0)=e_1$. 
We write $\langle\gamma(s_0), e_2 \rangle = p/(p-1) =:r$. 
Let $\phi_n:=1/n$ and $R_{\phi_n}\gamma|_{[s_0,L]}$, where $R_{\phi_n}$ denotes  the rotation matrix through angle $\phi_n$ about the line that passes through $\gamma(s_0)$ and is parallel to $e_1$: more precisely, 
\begin{align}\label{eq:rotation}
R_{\phi_n}\gamma|_{[s_0,L]}(s):= \gamma(s_0) + A_{\phi_n} \big( \gamma(s)-\gamma(s_0) \big), \quad s\in[s_0,L], 
\end{align}
where $A_{\phi_n} \in O(d)$ is the orthogonal matrix satisfying 
\begin{align*}
&(A_{\phi_n})_{ii}=1 \quad (i\neq2,3),  \hspace{31pt} (A_{\phi_n})_{23}=-\sin\phi_n, \quad (A_{\phi_n})_{32}= \sin\phi_n, \\
&(A_{\phi_n})_{22} =(A_{\phi_n})_{33} = \cos{\phi_n},  \quad (A_{\phi_n})_{ij}=0 \ (\text{otherwise}). 
\end{align*}

Now we define the curve $\Gamma_n\in C^1([0,L];\R^d)$ by 
\[
\Gamma_n:= \gamma|_{[0,s_0]} \oplus \big( R_{\phi_n}\gamma|_{[s_0,L]} \big)
\]
(see Figure~\ref{fig:3D-quasi} (ii)).
Let $\ell:=|\gamma(0)-\gamma(L)|$.
It follows from \eqref{eq:rotation} that 
\begin{align}\label{eq:end_Gamma_n}
\Gamma_n(L)= (-\ell, r(1-\cos\phi_n), -r\sin\phi_n, 0, \ldots, 0). 
\end{align}
Define the convergent sequence $\delta_n\to0$ by
\begin{align}\label{eq:delta_n}
\delta_n:=\ell - \sqrt{\ell^2-2r^2(1-\cos{\phi_n})} >0.
\end{align} 
Then there exists a large $n_0\in\N$ such that for all $n\geq n_0$ we have $\Gamma_n(\delta_n)=\gamma|_{[0,s_0]}(\delta_n)=-\delta_n e_1$, which combined with \eqref{eq:delta_n} and \eqref{eq:end_Gamma_n} yields
\[
|\Gamma_n(\delta_n) - \Gamma_n(L)| =\ell.
\]
Hence the translated restriction $P_0 \oplus \Gamma_n|_{[\delta_n,L]}$ 
satisfies the same boundary condition as $\gamma$, but the length is shorter since $\mathcal{L}[\Gamma_n|_{[\delta_n,L]}] = L-\delta_n$ (see Figure~\ref{fig:3D-quasi} (iii)).

In what follows we further modify $\{\Gamma_n|_{[\delta_n,L]}\}_{n\geq n_0}$ to make the length to be $L$, while preserving the endpoints.
Let $\alpha^\pm_n:[0,\frac{\delta_n}{2}]\to\R^d$ be the segments of length $\frac{\delta_n}{2}$ defined by $\alpha^\pm_n(s):= \pm s e_1$. 
For each $n\geq n_0$ we define $\gamma_n$ by 
\[
\gamma_n := \big( \alpha^-_n|_{[0,\frac{\delta_n}{2}]} \big) \oplus \big( \Gamma_n|_{[\delta_n,s_0]} \big) \oplus \big( \alpha^+_n|_{[0,\frac{\delta_n}{2}]} \big) \oplus \big( \Gamma_n|_{[s_0,L]} \big)
\]
(see Figure~\ref{fig:3D-quasi} (iv)). 
It then follows that $\mathcal{L}[\gamma_n]=L$ and $|\gamma_n(L)-\gamma_n(0)|=|\Gamma_n(L)-\Gamma_n(\delta_n)|=\ell$.
Noting also that $\gamma_n\in W^{2,p}(0,L;\R^d)$, we find that the translated curve $P_0 \oplus \gamma_n$ ($= \gamma_n+\delta_ne_1$) lies in $\Ap$ for all $n\geq n_0$.
Therefore, in view of Lemma~\ref{lem:Crelle-3D}, the proof is reduced to showing that $\{\gamma_n\}_{n\geq n_0}$ satisfies \eqref{eq:C}.

It is easy to verify the convergence \eqref{eq:C}-(i) as all the above procedures are $W^{2,p}$-small. 
Also, our construction is based only on cut-and-paste procedures together with deleting/inserting segments, the $p$-bending energy is preserved so that $\B_p[\gamma_n]=\B_p[\gamma]$, yielding \eqref{eq:C}-(ii). 
Finally, \eqref{eq:C}-(iii) follows from the fact that the vertex of the loop has non-vanishing curvature by \eqref{eq:loop_curvature_vector}, resulting the jump of the normal direction due to the rotation by $A_{\phi_n}$.
The proof is thus complete.
\end{proof}

\begin{figure}[htbp]
\centering
\includegraphics[width=90mm]{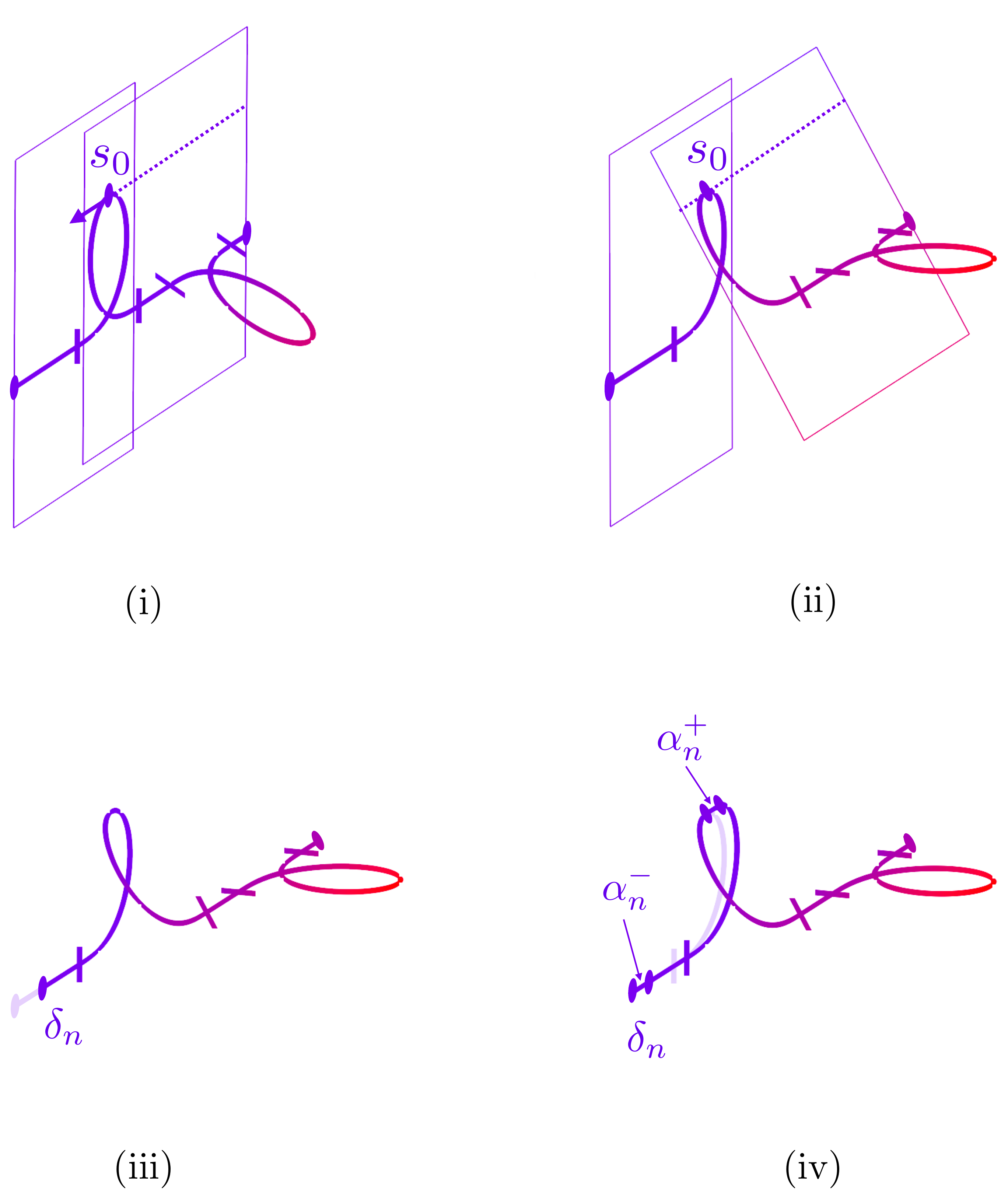} 
\caption{(i) A non-planar flat-core pinned $p$-elastica. 
(ii) Partly rotate at the vertex of the first loop. 
(iii) Cut the first $\delta_n$-segment to adjust the distance of the endpoints.
(iv) Insert two line segments $\alpha_n^\pm$ to adjust the total length without moving the endpoints.
}
\label{fig:3D-quasi}
\end{figure}

Now we turn to the case of $L_1=0$. 
Recall from \cite[Theorem 1.6]{GMarXiv2501} that if $\gamma:[0,L]\to\R^d$ is a pinned $p$-elastica, then the scalar curvature $k$ of $\gamma$ must satisfy $k(0)=k(L)=0$. 
Therefore, in the same spirit as Lemma~\ref{lem:Crelle-3D}, it suffices to construct a sequence $\{\gamma_n\}_{n\in\N} \subset \Ap$ satisfying the following three conditions: 
\begin{align}\tag{C${}_{\rm p}$}\label{eq:Cp}
\begin{cases}
\mathrm{(i)} \ \ \ \ \gamma_n \to \gamma \quad  \text{ in } \ \ W^{2,p}(0,L;\R^d) \ \  & \text{as} \ \  n\to \infty, \\
\mathrm{(ii)} \ \ \ \, \mathcal{B}_p[\gamma_n] \leq \mathcal{B}_p[\gamma] &  \text{for all (large)}\ n\in \N, \\
\mathrm{(iii)} \ \ \, \text{$k_n$ satisfies $k_n(0)\neq0$ or $k_n(L)\neq0$} & \text{for all (large)}\ n\in \N,
\end{cases}
\end{align}
where $k_n$ denotes the scalar curvature of $\gamma_n$.
To be more precise, condition (iii) in \eqref{eq:Cp} should be checked only if $\gamma_n$ is of class $C^2$: otherwise we may interpret that it is automatically valid, cf.\ \eqref{eq:C}-(iii).

\begin{proposition}\label{prop:instability-NOTquasi-3D}
Let $d\geq3$ and $\gamma\in \Ap$ be a flat-core pinned $p$-elastica. 
Suppose that, up to similarity and reparametrization, $\gamma$ is of the form \eqref{eq:N-loop-flat-core} with $L_1=0$.
Then $\gamma$ is unstable in $\Ap$.
\end{proposition}

The proof is completely parallel to the planar case \cite{MY_Crelle}.
Here we briefly sketch the argument for the reader's convenience.

\begin{proof}[Proof of Proposition~\ref{prop:instability-NOTquasi-3D}]
Without loss of generality we may assume that $\gamma$ is exactly given by \eqref{eq:N-loop-flat-core} with $L_1=0$.
Suppose on the contrary that $\gamma$ is stable.
For each integer $n\geq n_0$ with some $n_0$ such that $\frac{1}{n_0}<L$, we define $\gamma_n:[0,L]\to\R^d$ by
\begin{align}\label{eq:pinned-perturb}
\gamma_n:= P_0 \oplus \gamma|_{[\frac{1}{n},L]} \oplus \gamma|_{[0,\frac{1}{n}]}.     
\end{align}
Then $\{\gamma_n\}_{n\geq n_0} \subset \Ap$, but satisfies all the conditions in \eqref{eq:Cp}.
In particular, each $\gamma_n$ is of class $C^2$ and has the same energy as $\gamma$ but the vanishing-curvature boundary condition is lost.
See \cite[Proposition 6.4]{MY_Crelle} for details.
\end{proof}

\begin{proof}[Proof of Theorem~\ref{thm:nD-stability}]
It directly follows from Propositions~\ref{prop:instability-quasi-3D} and \ref{prop:instability-NOTquasi-3D}.
\end{proof}

\begin{proof}[Proof of Corollary~\ref{cor:nD-stable-unique}]
By \cite{GMarXiv2501}*{Theorem 1.6} any critical point is either a flat-core pinned $p$-elastica, or a non-flat-core planar pinned $p$-elastica.
The instability in the flat-core case follows by Theorem~\ref{thm:nD-stability}, while in the non-flat-core planar case, Corollary~\ref{cor:2D-stability} implies that the only stable critical point is a global minimizer.
\end{proof}



\bibliography{ref_Miura-Yoshizawa-ver250414}

\end{document}